\documentclass{amsart}
\usepackage{amsmath,amssymb}

\newtheorem{theorem}{Theorem}[section]
\newtheorem{lemma}[theorem]{Lemma}

\theoremstyle{definition}
\newtheorem{example}[theorem]{Example}
\newtheorem{conjecture}[theorem]{Conjecture}

\theoremstyle{remark}
\newtheorem{remark}[theorem]{Remark}

\numberwithin{equation}{section}

\newcommand{\mdeg}{\mathrm{mdeg}}
\newcommand{\Aut}{\mathrm{Aut}}
\newcommand{\Tame}{\mathrm{Tame}}

\begin{document}

\title{Multidegrees of Tame automorphisms with one prime number}

\author{Jiantao Li}
\address{School of Mathematics, Jilin university, 130012, Changchun, China} \email{jtlimath@gmail.com}

\author{Xiankun Du}
\address{School of Mathematics, Jilin university, 130012, Changchun, China} \email{duxk@jlu.edu.cn}

\date{\today}
\thanks{This research was supported by
NSF of China (No.11071097, No.11101176) and ``211 Project" and ``985 Project" of Jilin University.}

\begin{abstract}
Let $3\leq d_1\leq d_2\leq d_3$ be integers.  We show the following
results: (1) If $d_2$ is a prime number and
$\frac{d_1}{\gcd(d_1,d_3)}\neq2$, then $(d_1,d_2,d_3)$ is a
multidegree of a tame automorphism if and only if $d_1=d_2$ or
$d_3\in d_1\mathbb{N}+d_2\mathbb{N}$; (2) If $d_3$ is a prime number
and $\gcd(d_1,d_2)=1$, then $(d_1,d_2,d_3)$ is a multidegree of
a tame automorphism  if and only if $d_3\in
d_1\mathbb{N}+d_2\mathbb{N}$. We also relate this investigation with
a conjecture of Drensky and Yu,  which concerns with the lower bound
of the degree of the Poisson bracket of two polynomials,  and we
give a counter-example to this conjecture.
\end{abstract}

\maketitle


\section{Introduction}

Throughout this paper, let $F=(F_1,\dots,F_n): k^n\rightarrow k^n$
be a polynomial map, where $k$ is a field of characteristic $0$.  Denote by $\Aut~k^n$ the group of all
polynomial automorphisms of $k^n$.
Denote by $\mdeg F:=(\deg F_1,\ldots,\deg F_n)$ the {\it
multidegree} of $F$ and by $\mdeg$ the mapping from
the set of all polynomial maps into the set $\mathbb{N}^n$, where
$\mathbb{N}$ denotes the set of all nonnegative integers.

A polynomial automorphism $F=(F_1,\ldots,F_n)$ of $k^n$ is called
{\it elementary} if $$F=(x_1,\ldots,x_{i-1}, \alpha
x_i+f(x_1,\ldots,x_{i-1},x_{i+1},\ldots,x_n), x_{i+1},\ldots,x_n)$$
for $\alpha\in k^*$. Denote by $\Tame~k^n$ the subgroup of
$\Aut~k^n$ that is generated by all elementary automorphisms. The
element in $\Tame~k^n$ is called {\it tame automorphism}.  The
classical Jung-van der Kulk theorem \cite{Jung, Kulk} shows that
every polynomial automorphism of $k^2$ is tame. For many years
people believe that $\Aut~k^n$ is equal to $\Tame~k^n$. However, in
2004, Shestakov and Umirbaev \cite{SU04P,SU04T} proved the famous
Nagata conjecture, that is, the Nagata automorphism on $k^3$ is not
tame.

The multidegree plays an important role in the description of
polynomial automorphisms. For example, the Jacobian conjecture is
equivalent to the assert that if $(F_1, F_2)$ is a polynomial map
satisfying the Jacobian condition, then $\mdeg F=(\deg F_1, \deg
F_2)$ is principal, that is, $\deg F_1\mid \deg F_2$ or $\deg
F_2\mid \deg F_1$ \cite{Abhyankar08a}. But it is difficult to
describe the multidegrees of polynomial maps in higher dimensions,
even in dimension three. Recently, Kara\'{s} present a series of
papers concerning with multidegrees of tame automorphisms in
dimension three, see \cite{Karas10a,Karas11c,Karas11a,KarasZ11}.

In \cite{Karas10a}, Kara\'{s} proposed the following conjecture.
\begin{conjecture}\cite[Conjecture 4.1]{Karas10a}\label{Conjecture of Karas}
Let $3\leq p_1\leq d_2\leq d_3$ be integers with $p_1$ a prime number. Then $(p_1,d_2,d_3)\in \mdeg (\Tame~k^3)$ if and only if $p_1\mid d_2$ or $d_3\in p_1\mathbb{N}+d_2\mathbb{N}$.
\end{conjecture}

In \cite{Karas11e}, Kara\'{s} showed that if $\frac{d_3}{d_2}\neq\frac32$ or  $\frac{d_3}{d_2}=\frac32$ and $d_2>2p_1-4$, then Conjecture \ref{Conjecture of Karas} is valid. In \cite{SunC}, Sun and Chen also proved that if one of the following conditions is satisfied: (i) $\frac{d_2}{\gcd(d_2,d_3)}\neq2$; (ii) $\frac{d_3}{\gcd(d_2,d_3)}\neq3$; (iii) $d_2>2p_1-5$, then Conjecture \ref{Conjecture of Karas} is true.

In this paper, we consider a variation of the conjecture of Kara\'{s}. Let $3\leq d_1\leq d_2\leq d_3$ be integers. We show the following results: (1) If $d_2$ is a prime number and $\frac{d_1}{\gcd(d_1,d_3)}\neq2$, then $(d_1,d_2,d_3)\in
\mdeg(\Tame~k^3)$ if and only if $d_1=d_2$ or $d_3\in d_1\mathbb{N}+d_2\mathbb{N}$; (2) If $d_3$ is a prime number and $\gcd(d_1,d_2)=1$, then $(d_1,d_2,d_3)\in
\mdeg(\Tame~k^3)$ if and only if $d_3\in d_1\mathbb{N}+d_2\mathbb{N}$. We also relate this investigation to a conjecture of Drensky and Yu, which concerns with the lower bound of the degree of the Poisson bracket of two polynomials, and we give a counter-example of Drensky and Yu's conjecture.

\section{Preliminaries}

Recall that a pair $f, g \in \mathbb{C}[x_1,\ldots,x_n]$ is called $*$-{\it reduced} in \cite{SU04P,SU04T} if
\begin{enumerate}
\item $f, g$ are algebraically independent;
\item $\bar{f}, \bar{g}$ are algebraically dependent, where $\bar{f}$ denotes the highest homogeneous component of $f$;
\item $\bar{f}\notin \langle \bar{g} \rangle$ and $\bar{g}\notin \langle \bar{f} \rangle$.
\end{enumerate}

The following inequality plays an important role in the proof of the Nagata conjecture in \cite{SU04P,SU04T} and is also essential in our proofs.

\begin{theorem}\emph{(\cite[Theorem 3]{SU04P})}.\label{inequality}
Let $f, g \in k[x_1,\ldots,x_n]$ be a $*$-reduced pair, and $G(x,y)\in k[x,y]$ with $\deg_y G(x,y)=pq+r,\ 0\leq r<p$, where $p=\frac{\deg f}{\gcd(\deg f, \deg g)}$. Then
$$\deg G(f,g)\geq q(p\deg g-\deg f-\deg g+\deg [f,g])+r \deg g.$$
\end{theorem}

Note that $[f,g]$ means the Poisson bracket of $f$ and $g$ defined by
$$[f,g]=\sum_{1\leq i<j\leq n}(\frac{\partial f}{\partial x_i}\frac{\partial g}{\partial x_j}-\frac{\partial f}{\partial x_j}\frac{\partial g}{\partial x_i})[x_i,x_j].$$
By definition, $\deg [x_i,x_j]=2$ for $i\neq j$ and $\deg 0=-\infty$,
$$\deg [f,g]=\max_{1\leq i<j\leq n}\deg \{(\frac{\partial f}{\partial x_i}\frac{\partial g}{\partial x_j}-\frac{\partial f}{\partial x_j}\frac{\partial g}{\partial x_i})[x_i,x_j]\}.$$
It is shown that $[f,g]=0$ if and only if $f,g$ are algebraically dependent. If $f, g$ are algebraically independent, then
$$\deg [f,g]=2+\max_{1\leq i<j\leq n}\deg(\frac{\partial f}{\partial x_i}\frac{\partial g}{\partial x_j}-\frac{\partial f}{\partial x_j}\frac{\partial g}{\partial x_i})\geq2.$$

\begin{remark}\label{inequality2}
It is easy to shown (see \cite{Karas11a} for example) that Theorem \ref{inequality} is true even if $f,g$ just satisfy:
(1) $f, g$ are algebraically independent; (2) $\bar{f}\notin \langle \bar{g} \rangle$ and $\bar{g}\notin \langle \bar{f} \rangle$.
\end{remark}
\begin{theorem}{\rm (\cite[Theorem 2]{SU04T})}.\label{SUtheorem}
Let $F=(F_1,F_2,F_3)$ be a tame automorphism of $k^3$. If $\deg F_1+\deg F_2+\deg F_3>3$,  then $F$ admits either an elementary reduction or a reduction of types \textup{I-IV} \emph{(}see \cite[Definitions 1-4]{SU04T}\emph{)}.
\end{theorem}

\begin{remark}\label{reductionIV}
It is shown by Kuroda that there is no tame automorphism on $k[x,y,z]$ admitting reductions of type IV, see \cite[Theorem 7.1]{Kuroda10}.
\end{remark}

Recall that we say a polynomial
automorphism $F=(F_1,F_2,F_3)$ admits an {\it elementary reduction} if
there exists a polynomial $g\in k[x,y]$ and a permutation
$\sigma$ of the set $\{1,2,3\}$ such that $\deg
(F_{\sigma(1)}-g(F_{\sigma(2)},F_{\sigma(3)}))<\deg F_{\sigma(1)}$.

In this paper, we consider when $(d_1,d_2,d_3)$ is a multidegree of
a tame automorphism on $k^3$. Note that if $(F_1,F_2,F_3)$ with
multidegree $(d_1,d_2,d_3)$ is a tame automorphism, then after a
permutation $\sigma$, $(F_{\sigma(1)},F_{\sigma(2)},F_{\sigma(3)})$
is also a tame automorphism. It is also shown that if $d_1<3$, then
$(d_1,d_2,d_3)\in \mdeg(\Tame~k^3)$, see \cite[Corollary
3]{Karas11a}. Thus, without loss of generality, we can assume that
$3\leq d_1\leq d_2\leq d_3$.

\section{multidegree $(d_1,p_2,d_3)$ with $p_2$ a prime number}

In this section, let $3\leq d_1\leq p_2\leq d_3$ be integers with $p_2$ a prime number.
We start with some lemmas.

\begin{lemma}\cite{Brauer}\label{ineq}
If $a$, $b$ are positive integers that $\gcd(a,b)=1$, then $l\in a\mathbb{N}+b\mathbb{N}$ for all integers $l\geq (a-1)(b-1)$.
\end{lemma}

\begin{lemma}\label{elementary reduction}
If $(d_1,p_2,d_3)\in \mdeg(\Tame~k^3)$, then there exists a tame automorphism with multidegree $(d_1,p_2,d_3)$ which admits an elementary reduction.
\end{lemma}
\begin{proof}
Let $F$ be a tame automorphism with $\mdeg~F=(d_1,p_2,d_3)$. By Theorem \ref{SUtheorem} and Remark \ref{reductionIV}, $F$ admits an elementary reduction or a reduction of types I-III.

If $F$  admits a reduction of type III, then after a permutation, by \cite[Definition 3]{SU04T} there exists $n\in \mathbb{N}$ such that
\begin{align*}
&(3.1)\quad n<d_1\leq\frac32n,\ p_2=2n,\ d_3=3n;\quad \text{or}\\
&(3.2)\quad d_1=\frac32n,\ p_2=2n,\
\frac{5n}{2}<d_3\leq3n.
\end{align*}
Since $p_2$ is a prime number greater than $3$, (3.1) and (3.2) can not be satisfied. Thus, $F$ admits no reduction of type III.

By the definitions of reductions of types I and II, or see \cite[Proposition 20]{Karas11e}, if $F$ admits a reduction of type I or II, then there exists a tame automorphism admitting an elementary reduction with the same multidegree.
\end{proof}

We are now in a position to show our main result in this section.

\begin{theorem}\label{main of p2}
Let $3\leq d_1\leq p_2\leq d_3$ be integers with $p_2$ a prime number. If $\frac{d_1}{\gcd(d_1,d_3)}\neq2$, then $(d_1,p_2,d_3)\in
\mdeg(\Tame~k^3)$ if and only if $d_1=p_2$ or $d_3\in d_1\mathbb{N}+p_2\mathbb{N}$.
\end{theorem}

\begin{proof}
(1) If $d_1=p_2$ or $d_3\in d_1\mathbb{N}+p_2\mathbb{N}$,  then by \cite[Proposition 2.2]{Karas11a}, $(d_1,p_2,d_3)\in \mdeg(\Tame~k^3)$.

(2) Now suppose that $d_1<p_2$ and $d_3\notin d_1\mathbb{N}+p_2\mathbb{N}$. Moreover, by Lemma \ref{ineq}, $d_3<(d_1-1)(p_2-1)$. If there exists a tame automorphism with multidegree $(d_1,p_2,d_3)$, then by Lemma \ref{elementary reduction}, there exists a tame automorphism $F=(F_1,F_2,F_3)$  with $\mdeg~F=(d_1,p_2,d_3)$ admitting an elementary reduction.
Now the proof proceeds into three cases.

\textbf{Case 1:} If $F$ admits an elementary reduction of the form
$(F_1,F_2,F_3-g(F_1,F_2))$ such that $\deg(F_3-g(F_1,F_2))<\deg F_3$, then $\deg F_3=\deg g(F_1,F_2)$. Since $F$ is a polynomial automorphism, it follows that $F_i, F_j$ $(i,j=1,2,3)$ are algebraically independent, and hence $\deg[F_i,F_j]\geq 2$. Moreover, $\bar{F_i}\notin \langle \bar{F_j} \rangle$ since otherwise we have $\deg F_i \mid\deg F_j$, which contradicts to the fact that $d_1\nmid p_2$ and $d_3\notin d_1\mathbb{N}+p_2\mathbb{N}$. Note that $p=\frac{\deg F_1}{\gcd(\deg F_1,\deg F_2)}=d_1$. Set $\deg_y g(x,y)=d_1q+r,\ 0\leq r<d_1$. By Theorem \ref{inequality} and Remark \ref{inequality2},
\begin{align*}
d_3&=\deg F_3=\deg g(F_1,F_2)\\
   &\geq q(d_1p_2-d_1-p_2+\deg[F_1,F_2])+rp_2\\
   &\geq q(d_1p_2-d_1-p_2+2)+rp_2.
\end{align*}
Since $d_3<(d_1-1)(p_2-1)$, we have $q=0$.
Note that $0\leq r<d_1$, we can suppose that $g(x,y)=g_0(x)+g_1(x)y+\cdots+g_{d_1-1}(x)y^{d_1-1}$.
It follows from $\gcd(d_1,p_2)=1$ that the sets $d_1\mathbb{N},~ d_1\mathbb{N}+p_2,~\dots,~d_1\mathbb{N}+(d_1-1)p_2$ are disjoint. Thus,
\begin{align*}
d_3&=\deg g(F_1,F_2)=\deg(g_0(F_1)+g_1(F_1)F_2+\cdots+g_{d_1-1}(F_1)F_2^{d_1-1})\\
   &=\max_{0\leq i\leq d_1-1}\{\deg F_1\deg g_i+i\deg F_2\}=\max_{0\leq i\leq d_1-1}\{d_1\deg g_i+ip_2\},
\end{align*}
which contradicts $d_3\notin d_1\mathbb{N}+p_2\mathbb{N}$.

\textbf{Case 2:} If $F$ admits an elementary reduction of the form
$(F_1-g(F_2,F_3),F_2,F_3)$, then $\deg F_1=\deg g(F_2,F_3)$. $p=\frac{\deg F_2}{\gcd(\deg F_2,\deg F_3)}=p_2$. Set $\deg_y g(x,y)=p_2q+r,\ 0\leq r<p_2$. Then
\begin{align*}
d_1&=\deg F_1=\deg g(F_2,F_3)\\
   &\geq q(p_2d_3-p_2-d_3+\deg[F_2,F_3])+rd_3\\
   &\geq q(3d_3-p_2-d_3+2)+rd_3\\
   &\geq q((d_3-p_2)+d_3+2)+rd_3.
   \end{align*}
Since $d_1<(d_3-p_2)+d_3+2$ and $d_1<d_3$, it follows that $q=r=0$.
Suppose that $g(F_2,F_3)=g_1(F_2)$. Then $d_1=\deg F_1=\deg g_1(F_2)\in
p_2\mathbb{N}$, contrary to $d_1<p_2$.

\textbf{Case 3:} If $F$ admits an elementary reduction of the form
$(F_1,F_2-g(F_1,F_3),F_3)$, then $\deg F_2=\deg g(F_1,F_3)$.
It follows from $d_3\notin d_1\mathbb{N}+p_2\mathbb{N}$ that $\gcd(d_1,d_3)\neq
d_1$, whence $p=\frac{d_1}{\gcd(d_1,d_3)}\geq2$. Moreover, since $\frac{d_1}{\gcd(d_1,d_3)}\neq2$,  $p\geq3$. Let $\deg_y
g(x,y)=pq+r,\ 0\leq r<p$. Then
\begin{align*}
     p_2&=\deg F_2=\deg g(F_1,F_3)\\
        &\geq q(pd_3-d_1-d_3+\deg[F_1,F_3])+rd_3\\
        &\geq q(3d_3-d_1-d_3+2)+rd_3\\
        &=q((d_3-d_1)+d_3+2)+rd_3.
    \end{align*}
Thus, $q=r=0$.
Suppose that $g(F_1,F_3)=g_1(F_1)$. Then $p_2=\deg F_2=\deg g_1(F_1)\in
d_1\mathbb{N}$, a contradiction.

Therefore, $F$ can not admit any elementary reduction, the contradiction implies that there exists no tame automorphism with multidegree $(d_1,p_2,d_3)$ if $d_1<p_2$ and $d_3\notin d_1\mathbb{N}+p_2\mathbb{N}$.
\end{proof}

We claim that the condition $\frac{d_1}{\gcd(d_1,d_3)}\neq2$ in Theorem \ref{main of p2}  can not be removed. Indeed, Kuroda construct some tame automorphisms, after a permutation, with multidegree $(2m,2pm+p+1,(2p+1)m)$ admitting reductions of type I in \cite{Kuroda09}. Particularly, let $p=2$, $m=5$ or $11$. Then
\begin{example}\label{example}
There exist tame automorphisms with multidegree $(10,23,25)$ and $(22,47,55)$. Moreover, using the method in \cite{Kuroda09}, we can get a tame automorphism $F=(f_1,f_2,f_3)$ with $\mdeg~F=(10,23,25)$ admitting reductions of type I, where
$$\left\{
  \begin{array}{ll}
    f_1=x+y^2-g^2, \\
    f_2=\frac{256}{25}f_1^5+g+h^2, \\
    f_3=f_2+h,
  \end{array}
\right.
$$
$g=z+3x^2y+3xy^3+y^5$ and $h=y-6(x+y^2)^2g+8(x+y^2)g^3-\frac{16}5g^5$.
\end{example}

In the proof of Theorem \ref{main of p2}, we observe that if a more precise lower bound of $\deg[F_1,F_3]$ is given, then we can give a  better description of $\mdeg(\Tame~k^3)$.  This is closely related to a conjecture of Drensky and Jie-Tai Yu.
\begin{conjecture}\cite{DrenskyY09}\label{Yu}
Let $f$ and $g$ be algebraically independent polynomials in $k[x_1,\ldots,x_n]$ such that the homogeneous components of maximal degree of $f$ and $g$ are algebraically
dependent, $f$ and $g$ generate their integral closures $C(f)$ and $C(g)$ in $k[x_1,\ldots,x_n]$, respectively, and neither $\deg f | \deg g$ nor $\deg g | \deg f$. Then
$$\deg[f,g]>\min\{\deg(f), \deg(g)\}.$$
\end{conjecture}
Although some counter-examples of  Conjecture \ref{Yu} are given  in \cite{DrenskyY09}, it is still of great interest to find a meaningful
lower bound of $\deg[f,g]$, and such a bound will give a nice description of $\Tame~k^n$ and $\Aut~k^n$.  We observe that, from Example \ref{example}, we can construct some counter-examples of Conjecture \ref{Yu}.

\begin{example}\label{counterexample}
$F=(f_1,f_2,f_3)=(x+y^2-g^2,\frac{256}{25}f_1^5+g+h^2,h)$ is a tame automorphism admitting an elementary reduction, where $g=z+3x^2y+3xy^3+y^5$ and $h=y-6(x+y^2)^2g+8(x+y^2)g^3-\frac{16}5g^5$.
Moreover, $\mdeg~F=(10,23,25)$, $\deg [f_1,f_3]=8<\min \{\deg f_1,\deg f_3\}$. Thus, $(f_1,f_3)$ is a counter-example of Conjecture \ref{Yu}.
\end{example}
\begin{proof}
It follows from Example \ref{example} that $F'=(x+y^2-g^2,\frac{256}{25}f_1^5+g+h^2,f_2+h)$ is a tame automorphism with $\mdeg~F'=(10,23,25)$ admitting an reduction of type I. By~\cite[Proposition 1]{SU04T}, after composing an affine automorphism $(x,y,z-y)$, $F=(f_1,f_2,f_3)=(x+y^2-g^2, \frac{256}{25}f_1^5+g+h^2, h)$ is a tame automorphism with $\mdeg~F=(10,23,25)$ admitting an elementary reduction. Moreover,
\begin{multline*}
  \deg [f_1,f_3]= (-30x^2y^4-54x^3y^2-18x^4-6y^3z-12xyz+1)[x,y]\\
   -(6y^4+12xy^2+6x^2)[x,z]+(-10y^5-18xy^3-6x^2y+2z)[y,z].
\end{multline*}
Thus, $\deg [f_1,f_3]=8<\min \{\deg f_1,\deg f_3\}$, whence the homogeneous components of maximal degree of $f_1$ and $f_3$ are algebraically
dependent. Furthermore,
since $(f_1,f_2,f_3)$ is a polynomial automorphism, it follows that $f_1$ and $f_3$ are algebraically independent irreducible polynomials.  Thus, $(f_1,f_3)$ is a counter-example of Conjecture \ref{Yu}.
\end{proof}

\section{multidegree $(d_1,d_2,p_3)$ with $p_3$ a prime number}

In this section, let $3\leq d_1\leq d_2\leq p_3$ be integers with $\gcd(d_1,d_2)=1$ and $p_3$ a prime number.

\begin{lemma}\label{elementary reduction of p3}
If $(d_1,d_2,p_3)\in \mdeg(\Tame~k^3)$ with $\gcd(d_1,d_2)=1$ and $p_3$ a prime number, then there exists a tame automorphism with multidegree $(d_1,d_2,p_3)$ which admits an elementary reduction.
\end{lemma}
\begin{proof}
Let $F$ be a tame automorphism with $\mdeg~F=(d_1,d_2,p_3)$. By Theorem \ref{SUtheorem} and Remark \ref{reductionIV}, $F$ admits an elementary reduction or a reduction of types I-III.

If $F$  admits a reduction of type III, then after a permutation, by \cite[Definition 3]{SU04T} there exists $n\in \mathbb{N}$ such that
\begin{align*}
&(4.1)\quad n<d_1\leq\frac32n,\ d_2=2n,\ p_3=3n;\quad \text{or}\\
&(4.2)\quad d_1=\frac32n,\ d_2=2n,\
\frac{5n}{2}<p_3\leq3n.
\end{align*}
Since $p_3$ is a prime number greater that $3$, (4.1) can not be satisfied.
If $(d_1,d_2,p_3)$ satisfies (4.2), it follows from $\gcd(d_1,d_2)=1$  that $n=2$, and hence $5<p_3\leq6$, contrary to the fact that $p_3$ is a prime number. Thus, $F$ admits no reduction of type III.

By \cite[Proposition 20]{Karas11e}, if $F$ admits a reduction of type I or II, then there exists a tame automorphism with the same multidegree that admits an elementary reduction.
\end{proof}

We can now formulate our main result in this section.

\begin{theorem}\label{main of p3}
Let $3\leq d_1\leq d_2\leq p_3$ be integers with $\gcd(d_1,d_2)=1$ and $p_3$ a prime number. Then $(d_1,d_2,p_3)\in
\mdeg(\Tame~k^3)$ if and only if $p_3\in d_1\mathbb{N}+d_2\mathbb{N}$.
\end{theorem}

\begin{proof}
(1) If $p_3\in d_1\mathbb{N}+d_2\mathbb{N}$, by \cite[Proposition 2.2]{Karas11a}, $(d_1,d_2,p_3)\in \mdeg(\Tame~k^3)$.

(2) Now suppose that $p_3\notin d_1\mathbb{N}+d_2\mathbb{N}$, whence  $p_3<(d_1-1)(d_2-1)$ by Lemma \ref{ineq}.  If there exists a tame automorphism with multidegree $(d_1,d_2,p_3)$, then by Lemma \ref{elementary reduction of p3}, there exists a tame automorphism $F=(F_1,F_2,F_3)$  with $\mdeg~F=(d_1,d_2,p_3)$ admitting an elementary reduction.
Now the proof will be divided into three cases.

\textbf{Case 1:} If $F$ admits an elementary reduction of the form
$(F_1-g(F_2,F_3),F_2,F_3)$, then $\deg F_1=\deg g(F_2,F_3)$. $p=\frac{\deg F_2}{\gcd(\deg F_2,\deg F_3)}=d_2$. Set $\deg_y g(x,y)=d_2q+r,\ 0\leq r<d_2$. Then
\begin{align*}
d_1&=\deg F_1=\deg g(F_2,F_3)\\
   &\geq q(d_2p_3-d_2-p_3+\deg[F_2,F_3])+rp_3\\
   &\geq q(3p_3-d_2-p_3+2)+rp_3\\
   &\geq q((p_3-d_2)+p_3+2)+rp_3.
   \end{align*}
Thus, $q=r=0$.
Suppose that $g(F_2,F_3)=g_1(F_2)$. Then $d_1=\deg F_1=\deg g_1(F_2)\in
d_2\mathbb{N}$, contrary to $d_1<d_2$.

\textbf{Case 2:}  If $F$ admits an elementary reduction of the form
$(F_1,F_2-g(F_1,F_3),F_3)$, then $\deg F_2=\deg g(F_1,F_3)$.
$p=\frac{\deg F_1}{\gcd(\deg F_1,\deg F_3)}=d_1$. Set $\deg_y
g(x,y)=d_1q+r,\ 0\leq r<d_1$. Then
\begin{align*}
     d_2&=\deg F_2=\deg g(F_1,F_3)\\
        &\geq q(d_1p_3-d_1-p_3+\deg[F_1,F_3])+rp_3\\
        &\geq q(3p_3-d_1-p_3+2)+rp_3\\
        &=q((p_3-d_1)+p_3+2)+rp_3.
    \end{align*}
Thus, $q=r=0$.
Suppose that $g(F_1,F_3)=g_1(F_1)$. Then $d_2=\deg F_2=\deg g_1(F_1)\in
d_1\mathbb{N}$, a contradiction.

\textbf{Case 3:} If $F$ admits an elementary reduction of the form
$(F_1,F_2,F_3-g(F_1,F_2))$ such that $\deg(F_3-g(F_1,F_2))<\deg F_3$, then $\deg F_3=\deg g(F_1,F_2)$. It follows from $\gcd(d_1,d_2)=1$ that $p=\frac{\deg F_1}{\gcd(\deg F_1,\deg F_2)}=d_1$. Set $\deg_y g(x,y)=d_1q+r,\ 0\leq r<d_1$. Then
\begin{align*}
p_3&=\deg F_3=\deg g(F_1,F_2)\\
   &\geq q(d_1d_2-d_1-d_2+\deg[F_1,F_2])+rd_2\\
   &\geq q(d_1d_2-d_1-d_2+2)+rd_2.
\end{align*}
Since $p_3<(d_1-1)(d_2-1)$, we have $q=0$.
Note that $0\leq r<d_1$, we can suppose that $g(x,y)=g_0(x)+g_1(x)y+\cdots+g_{d_1-1}(x)y^{d_1-1}$.
It follows from $\gcd(d_1,d_2)=1$ that the sets $d_1\mathbb{N},~ d_1\mathbb{N}+d_2,~\dots,~d_1\mathbb{N}+(d_1-1)d_2$ are disjoint. Thus,
\begin{align*}
p_3&=\deg g(F_1,F_2)=\deg(g_0(F_1)+g_1(F_1)F_2+\cdots+g_{d_1-1}(F_1)F_2^{d_1-1})\\
   &=\max_{0\leq i\leq d_1-1}\{\deg F_1\deg g_i+i\deg F_2\}=\max_{0\leq i\leq d_1-1}\{d_1\deg g_i+id_2\},
\end{align*}
which contradicts $p_3\notin d_1\mathbb{N}+d_2\mathbb{N}$.

Thus, $F$ admits no elementary reduction, the contradiction implies that there exists no tame automorphism with multidegree $(d_1,d_2,p_3)$ if  $p_3\notin d_1\mathbb{N}+d_2\mathbb{N}$.
\end{proof}

\bibliographystyle{amsplain}

\end{document}